\documentclass[a4paper,11pt]{amsart}
\normalsize

\usepackage{amssymb}
\usepackage{bbm}
\usepackage[all]{xy}

\newtheoremstyle{bracket}{1ex}{2ex}{\rm}{}{\bfseries}{}{0.8em}{\thmnumber{(#2)}}
\newtheoremstyle{thm}{1ex}{2ex}{\itshape}{}{\bfseries}{}{0.9em}{\thmnumber{(#2)}\thmname{ #1}\thmnote{ \rm (#3)}}
\newtheoremstyle{example}{1ex}{2ex}{\rm}{}{\bfseries}{}{0.8em}{\thmnumber{(#2)}\thmname{ #1}}

\theoremstyle{bracket}
\newtheorem{no}{}[section]

\theoremstyle{thm}
\newtheorem{prop}[no]{Proposition}
\newtheorem{lemma}[no]{Lemma}
\newtheorem{cor}[no]{Corollary}
\newtheorem{thm}[no]{Theorem}

\theoremstyle{example}
\newtheorem{exas}[no]{Examples}

\newcommand{\grmod}{{\sf GrMod}}
\newcommand{\grann}{{\sf GrAnn}}
\newcommand{\ann}{{\sf Ann}}
\newcommand{\catmod}{{\sf Mod}}
\newcommand{\ab}{{\sf Ab}}
\newcommand{\hm}[3]{{\rm Hom}_{#1}(#2,#3)}
\newcommand{\grhm}[4]{{}^{#1}{\rm Hom}_{#2}(#3,#4)}
\newcommand{\ps}{[\psi]}
\newcommand{\Id}{{\rm Id}}
\newcommand{\Z}{\mathbbm{Z}}
\newcommand{\NN}{{\bf N}}
\newcommand{\N}{\mathbbm{N}}
\newcommand{\C}{{\sf C}}
\DeclareMathOperator{\ke}{Ker}
\newdir{ (}{!/5pt/@{ }*@^{(}}
\newdir{ >}{{}*!/-10pt/\dir{>}}

\newcommand{\sq}{\hskip2pt\raisebox{.225ex}{\rule{.8ex}{.8ex}\hskip2pt}}

\newcommand{\snf}{\renewcommand{\thefootnote}{*}\footnotetext{The author was supported by the Swiss National Science Foundation.}}

\begin{document}

\title[Coarsenings, injectives and Hom functors]{Coarsenings, injectives and Hom functors\snf}
\author{Fred Rohrer}
\address{Universit\"at T\"ubingen, Fachbereich Mathematik, Auf der Morgenstelle 10, 72076 T\"u\-bingen, Germany}
\email{rohrer@mail.mathematik.uni-tuebingen.de}
\subjclass[2010]{Primary 13A02; Secondary 13C11}
\keywords{Graded ring, graded module, coarsening, refinement, injective module, small module, steady ring, graded Hom functor}

\begin{abstract}
It is characterized when coarsening functors between categories of graded modules preserve injectivity of objects, and when they commute with graded covariant Hom functors.
\end{abstract}

\maketitle


\section*{Introduction}
Throughout, groups and rings are understood to be commutative. By $\ab$, $\ann$ and $\catmod(R)$ for a ring $R$ we denote the categories of groups, rings, and $R$-modules, respectively. If $G$ is a group, then by a $G$-graded ring we mean a pair $(R,(R_g)_{g\in G})$ consisting of a ring $R$ and a family $(R_g)_{g\in G}$ of subgroups of the additive group of $R$ whose direct sum equals the additive group of $R$ such that for $g,h\in G$ it holds $R_gR_h\subseteq R_{g+h}$. If $(R,(R_g)_{g\in G})$ is a $G$-graded ring, then by a $G$-graded $R$-module we mean a pair $(M,(M_g)_{g\in G})$ consisting of an $R$-module $M$ and a family $(M_g)_{g\in G}$ of subgroups of the additive group of $M$ whose direct sum equals the additive group of $M$ such that for $g,h\in G$ it holds $R_gM_h\subseteq M_{g+h}$. If no confusion can arise then we denote a $G$-graded ring $(R,(R_g)_{g\in G})$ just by $R$, and a $G$-graded $R$-module $(M,(M_g)_{g\in G})$ just by $M$. Accordingly, for a $G$-graded ring $R$, a $G$-graded $R$-module $M$ and $g\in G$ we denote by $M_g$ the component of degree $g$ of $M$. Given $G$-graded rings $R$ and $S$, by a morphism of $G$-graded rings from $R$ to $S$ we mean a morphism of rings $u:R\rightarrow S$ such that $u(R_g)\subseteq S_g$ for $g\in G$, and given a $G$-graded ring $R$ and $G$-graded $R$-modules $M$ and $N$, by a morphism of $G$-graded $R$-modules from $M$ to $N$ we mean a morphism of $R$-modules $u:M\rightarrow N$ such that $u(M_g)\subseteq N_g$ for $g\in G$. We denote by $\grann^G$ and $\grmod^G(R)$ for a $G$-graded ring $R$ the categories of $G$-graded rings and $G$-graded $R$-modules, respectively, with the above notions of morphisms. In case $G=0$ we canonically identify $\grann^G$ with $\ann$ and $\grmod^G(R)$ with $\catmod(R)$ for a ring $R$. 

Let $\psi\colon G\twoheadrightarrow H$ be an epimorphism in $\ab$ and let $R$ be a $G$-graded ring. We consider the $\psi$-coarsening $R_{\ps}$ of $R$, i.e., the $H$-graded ring whose underlying ring is the ring underlying $R$ and whose component of degree $h\in H$ is $\bigoplus_{g\in\psi^{-1}(h)}R_g$. An analogous construction for graded modules yields the $\psi$-coarsening functor $\bullet_{\ps}\colon\grmod^G(R)\rightarrow\grmod^H(R_{\ps})$, coinciding for $H=0$ with the functor that forgets the graduation. The aim of this note is to study functors of this type.

Let $M$ be a $G$-graded $R$-module. Some properties of $M$ behave well under coarsening functors -- e.g., $M$ is projective (in $\grmod^G(R)$) if and only if $M_{\ps}$ is projective (in $\grmod^H(R_{\ps})$) --, but others do not. An example is injectivity. For $H=0$ it is well known that if $M_{\ps}$ is injective then so is $M$, but that the converse does not necessarily hold. However, the converse \textit{does} hold if $G$ is finite, as shown by N\v{a}st\v{a}sescu, Raianu and Van Oystaeyen (\cite{nro}). We generalize this to arbitrary $H$ by showing that the converse holds if $\ke(\psi)$ is finite, and we moreover show that this is the best possible without imposing further conditions on $R$ or $M$ (Theorem \ref{1.80}). One should note that finiteness of $\ke(\psi)$ is fulfilled if $G$ is of finite type and $\psi$ is the canonical projection onto $G$ modulo its torsion subgroup. Such coarsenings can be used to reduce the study of graduations by groups of finite type to that of (often easier) graduations by free groups of finite rank.

A further interesting question is whether coarsening functors commute with graded Hom functors. The $G$-graded covariant Hom functor $\grhm{G}{R}{M}{\bullet}$ of $M$ maps a $G$-graded $R$-module $N$ onto the $G$-graded $R$-module $$\grhm{G}{R}{M}{N}=\bigoplus_{g\in G}\hm{\grmod^G(R)}{M}{N(g)}$$ (where $\bullet(g)$ denotes shifting by $g$). There is a canonical monomorphism of functors $h^M_{\psi}\colon\grhm{G}{R}{M}{\bullet}_{\ps}\rightarrowtail\grhm{H}{R_{\ps}}{M_{\ps}}{\bullet_{\ps}}$. For $H=0$ this is an isomorphism if and only if $G$ is finite or $M$ is small, as shown by G\'omez Pardo, Militaru and N\v{a}st\v{a}sescu (\cite{gpmn}). We generalize this to arbitrary $H$ by showing that $h^M_{\psi}$ is an isomorphism if and only if $\ke(\psi)$ is finite or $M$ is small (Theorem \ref{2.70}). A surprising consequence is that if $h^M_{\psi}$ is an isomorphism for \textit{some} epimorphism $\psi$ with infinite kernel then $h^M_{\varphi}$ is an isomorphism for \textit{every} epimorphism $\varphi$ (Corollary \ref{2.80}).

The proofs of the aforementioned results are similar to and inspired by those in \cite{nro} and \cite{gpmn}. In particular, they partially rely on the existence of adjoint functors of coarsening functors, treated in the first section.


\section{Coarsening functors and their adjoints}\label{sec1}

\noindent\textit{Let $\psi\colon G\twoheadrightarrow H$ be an epimorphism in $\ab$.}\smallskip

We first recall the definition of coarsening and refinement functors for rings and modules, and the construction of some canonical morphisms of functors.

\begin{no}
A) For a $G$-graded ring $R$ there is an $H$-graded ring $R_{\ps}$ with underlying ring the ring underlying $R$ and with $H$-graduation $(\bigoplus_{g\in\psi^{-1}(h)}R_g)_{h\in H}$. For $u\colon R\rightarrow S$ in $\grann^G$ there is a $u_{\ps}\colon R_{\ps}\rightarrow S_{\ps}$ in $\grann^H$ with underlying map the map underlying $u$. This defines a functor $\bullet_{\ps}\colon\grann^G\rightarrow\grann^H$, called \textit{$\psi$-coarsening.}

\smallskip

B) For an $H$-graded ring $S$ there is a $G$-graded ring $S^{\ps}$ with $G$-graduation $(S_{\psi(g)})_{g\in G}$, so that its underlying additive group is $\bigoplus_{g\in G}S_{\psi(g)}$, and with multiplication given by the maps $S_{\psi(g)}\times S_{\psi(h)}\rightarrow S_{\psi(g)+\psi(h)}$ for $g,h\in G$ induced by the multiplication of $S$. For $v\colon S\rightarrow T$ in $\grann^H$ there is a $v^{\ps}\colon S^{\ps}\rightarrow T^{\ps}$ in $\grann^G$ with $v^{\ps}_g=v_{\psi(g)}$ for $g\in G$. This defines a functor $\bullet^{\ps}\colon\grann^H\rightarrow\grann^G$, called \textit{$\psi$-refinement.}
\end{no}

\begin{no}\label{1.26}
A) For a $G$-graded ring $R$, the coproduct in $\ab$ of the canonical injections $R_g\rightarrowtail\bigoplus_{f\in\psi^{-1}(\psi(g))}R_f=(R_{\ps})^{\ps}_g$ for $g\in G$ is a monomorphism $\alpha_{\psi}(R)\colon R\rightarrowtail(R_{\ps})^{\ps}$ in $\grann^G$, and the coproduct in $\ab$ of the restrictions $$(R_{\ps})^{\ps}_g=\bigoplus_{f\in\psi^{-1}(\psi(g))}R_f\twoheadrightarrow R_g$$ of the canonical projections $\prod_{f\in\psi^{-1}(\psi(g))}R_f\twoheadrightarrow R_g$ for $g\in G$ is an epimorphism $\delta_{\psi}(R)\colon(R_{\ps})^{\ps}\twoheadrightarrow R$ in $\grann^G$. Varying $R$ we get a monomorphism $\alpha_{\psi}\colon\Id_{\grann^G}\rightarrowtail(\bullet_{\ps})^{\ps}$ and an epimorphism $\delta_{\psi}\colon(\bullet_{\ps})^{\ps}\twoheadrightarrow\Id_{\grann^G}$.

\smallskip

B) For an $H$-graded ring $S$, the coproduct in $\ab$ of the codiagonals $$((S^{\ps})_{\ps})_h=\bigoplus_{g\in\psi^{-1}(h)}S_h\twoheadrightarrow S_h$$ for $h\in H$ is an epimorphism $\beta_{\psi}(S)\colon(S^{\ps})_{\ps}\twoheadrightarrow S$ in $\grann^H$. If $\ke(\psi)$ is finite, so that $\psi^{-1}(h)$ is finite for $h\in H$, then the coproduct in $\ab$ of the diagonals $S_h\rightarrowtail\prod_{g\in\psi^{-1}(h)}S_h=\bigoplus_{g\in\psi^{-1}(h)}S^{\ps}_g=((S^{\ps})_{\ps})_h$ for $h\in H$ is a monomorphism $\gamma_{\psi}(S)\colon S\rightarrowtail(S^{\ps})_{\ps}$ in $\grann^H$. Varying $S$ we get an epimorphism $\beta_{\psi}\colon(\bullet^{\ps})_{\ps}\twoheadrightarrow\Id_{\grann^H}$ and --- if $\ke(\psi)$ is finite --- a monomorphism $\gamma_{\psi}\colon\Id_{\grann^H}\rightarrowtail(\bullet^{\ps})_{\ps}$.
\end{no}

\begin{no}\label{1.25}
A) Let $R$ be a $G$-graded ring. For a $G$-graded $R$-module $M$ there is an $H$-graded $R_{\ps}$-module $M_{\ps}$ with underlying $R_{[0]}$-module the $R_{[0]}$-module underlying $M$ and with $H$-graduation $(\bigoplus_{g\in\psi^{-1}(h)}M_g)_{h\in H}$. For $u\colon M\rightarrow N$ in $\grmod^G(R)$ there is a $u_{\ps}\colon M_{\ps}\rightarrow N_{\ps}$ in $\grmod^H(R_{\ps})$ with underlying map the map underlying $u$. This defines an exact functor $$\bullet_{\ps}\colon\grmod^G(R)\rightarrow\grmod^H(R_{\ps}),$$ called \textit{$\psi$-coarsening.}

\smallskip

B) Let $S$ be an $H$-graded ring. For an $H$-graded $S$-module $M$ there is a $G$-graded $S^{\ps}$-module $M^{\ps}$ with $G$-graduation $(M_{\psi(g)})_{g\in G}$, so that its underlying additive group is $\bigoplus_{g\in G}M_{\psi(g)}$, and with $S^{\ps}$-action given by the maps $$S_{\psi(g)}\times M_{\psi(h)}\rightarrow M_{\psi(g)+\psi(h)}$$ for $g,h\in G$ induced by the $S$-action of $M$. For $u\colon M\rightarrow N$ in $\grmod^H(S)$ there is a $u^{\ps}\colon M^{\ps}\rightarrow N^{\ps}$ in $\grmod^G(S^{\ps})$ with $u^{\ps}_g=u_{\psi(g)}$ for $g\in G$. This defines an exact functor $\bullet^{\ps}\colon\grmod^H(S)\rightarrow\grmod^G(S^{\ps})$, called \textit{$\psi$-refinement.}

\smallskip

C) For a $G$-graded ring $R$, composing $$\bullet^{\ps}\colon\grmod^H(R_{\ps})\rightarrow\grmod^G((R_{\ps})^{\ps})$$ with scalar restriction $\grmod^G((R_{\ps})^{\ps})\rightarrow\grmod^G(R)$ by means of $\alpha_{\psi}(R)$ (\ref{1.26}~A)) yields an exact functor $\grmod^H(R_{\ps})\rightarrow\grmod^G(R)$, by abuse of language again denoted by $\bullet^{\ps}$ and called \textit{$\psi$-refinement.}
\end{no}

\begin{no}\label{1.36}
A) Let $R$ be a $G$-graded ring. For a $G$-graded $R$-module $M$, the coproduct in $\ab$ of the canonical injections $M_g\rightarrowtail\bigoplus_{f\in\psi^{-1}(\psi(g))}M_f=(M_{\ps})^{\ps}_g$ for $g\in G$ is a monomorphism $\alpha'_{\psi}(M)\colon M\rightarrowtail(M_{\ps})^{\ps}$ in $\grmod^G(R)$, and the coproduct in $\ab$ of the restrictions $(M_{\ps})^{\ps}_g=\bigoplus_{f\in\psi^{-1}(\psi(g))}M_f\twoheadrightarrow M_g$ of the canonical projections $\prod_{f\in\psi^{-1}(\psi(g))}M_f\twoheadrightarrow M_g$ for $g\in G$ is an epimorphism $\delta'_{\psi}(M)\colon(M_{\ps})^{\ps}\twoheadrightarrow M$ in $\grmod^G(R)$. Varying $M$ we get a monomorphism $\alpha'_{\psi}\colon\Id_{\grmod^G(R)}\rightarrowtail(\bullet_{\ps})^{\ps}$ and an epimorphism $\delta'_{\psi}\colon(\bullet_{\ps})^{\ps}\twoheadrightarrow\Id_{\grmod^G(R)}$.

\smallskip

B) For an $H$-graded $R_{\ps}$-module $M$, the coproduct in $\ab$ of the codiagonals $((M^{\ps})_{\ps})_h=\bigoplus_{g\in\psi^{-1}(h)}M_h\twoheadrightarrow M_h$ for $h\in H$ is an epimorphism $$\beta'_{\psi}(M)\colon(M^{\ps})_{\ps}\twoheadrightarrow M$$ in $\grmod^H(R_{\ps})$. If $\ke(\psi)$ is finite, so that $\psi^{-1}(h)$ is finite for $h\in H$, then the coproduct in $\ab$ of the diagonals $$M_h\rightarrowtail\prod_{g\in\psi^{-1}(h)}M_h=\bigoplus_{g\in\psi^{-1}(h)}M^{\ps}_g=((M^{\ps})_{\ps})_h$$ for $h\in H$ is a monomorphism $\gamma'_{\psi}(M)\colon M\rightarrowtail(M^{\ps})_{\ps}$ in $\grmod^G(R)$. Varying $M$ we get an epimorphism $\beta'_{\psi}\colon(\bullet^{\ps})_{\ps}\twoheadrightarrow\Id_{\grmod^H(R_{\ps})}$ and --- if $\ke(\psi)$ is finite --- a monomorphism $\gamma'_{\psi}\colon\Id_{\grmod^H(R_{\ps})}\rightarrowtail(\bullet^{\ps})_{\ps}$.
\end{no}

\begin{exas}\label{1.20}
A) If $H=0$ then $\bullet_{\ps}$ coincides with the forgetful functor $\grann^G\rightarrow\ann$ (for rings) or $\grmod^G(R)\rightarrow\catmod(R_{[0]})$ (for modules) that forgets the graduation.

\smallskip

B) Let $\psi\colon\Z/2\Z\rightarrow 0$ and let $S$ be a ring. The underlying additive group of $S^{\ps}$ is the group $S\oplus S$, its components of degree $\overline{0}$ and $\overline{1}$ are $S\times 0$ and $0\times S$, respectively, and its multiplication is given by $(a,b)(c,d)=(ac+bd,ad+cb)$ for $a,b,c,d\in S$.

\smallskip

C) Let $A$ be a ring and let $R$ be the $G$-graded ring with $R_0=A$ and $R_g=0$ for $g\in G\setminus 0$. Then, $R_{\ps}$ is the $H$-graded ring with $(R_{\ps})_0=A$ and $(R_{\ps})_h=0$ for $h\in H\setminus 0$, and $\grmod^G(R)$ and $\grmod^H(R_{\ps})$ are canonically isomorphic to the product categories $\catmod(A)^G$ and $\catmod(A)^H$, respectively. Under these isomorphisms, $\bullet_{\ps}$ and $\bullet^{\ps}$ correspond to functors $\catmod(A)^G\rightarrow\catmod(A)^H$ with $(M_g)_{g\in G}\mapsto(\bigoplus_{g\in\psi^{-1}(h)}M_g)_{h\in H}$ and $\catmod(A)^H\rightarrow\catmod(A)^G$ with $(M_h)_{h\in H}\mapsto(M_{\psi(g)})_{g\in G}$, respectively. Using this it is readily checked that for an $H$-graded $R_{\ps}$-module $M$ it holds $(M^{\ps})_{\ps}=M^{\oplus\ke(\psi)}$.
\end{exas}

For modules, $\psi$-coarsening is left adjoint to $\psi$-refinement (\cite[3.1]{nro}), and for $H=0$ the same holds for rings (\cite[1.2.2]{no2}). We recall now the result for modules and generalize the one for rings to arbitrary $H$.

\begin{prop}\label{1.37}
a) For a $G$-graded ring $R$ there is an adjunction $$\bigl(\grmod^G(R)\xrightarrow{\bullet_{\ps}}\grmod^H(R_{\ps}),\grmod^H(R_{\ps})\xrightarrow{\bullet^{\ps}}\grmod^G(R)\bigr)$$ with unit $\alpha'_{\psi}$ and counit $\beta'_{\psi}$.

b) There is an adjunction $$\bigl(\grann^G\xrightarrow{\bullet_{\ps}}\grann^H,\grann^H\xrightarrow{\bullet^{\ps}}\grann^G\bigr)$$ with unit $\alpha_{\psi}$ and counit $\beta_{\psi}$.
\end{prop}

\begin{proof}
Straightforward.
\end{proof}

In \cite[3.1]{nro} it was shown that if $\ke(\psi)$ is finite then $\psi$-refinement for modules is left adjoint to $\psi$-coarsening. For $H=0$ this was sharpened by the result that $\psi$-coarsening has a left adjoint if and only if $G$ is finite (\cite[2.5]{dnro}). We now generalize this to arbitrary $H$ and prove moreover the corresponding statement for rings. We will need the following remark on products of graded rings and modules.

\begin{no}\label{1.50}
A) The category $\grann^G$ has products, but $\bullet_{\ps}$ does not necessarily commute with them. The product $R=\prod_{i\in I}R^{(i)}$ of a family $(R^{(i)})_{i\in I}$ of $G$-graded rings in $\grann^G$ is a $G$-graded ring as follows. Its components are $\prod_{i\in I}R^{(i)}_g$ for $g\in G$, so that its underlying additive group is $\bigoplus_{g\in G}\prod_{i\in I}R^{(i)}_g$. For $i\in I$, the multiplication of $R^{(i)}$ is given by maps $R^{(i)}_g\times R^{(i)}_h\rightarrow R^{(i)}_{g+h}$ for $g,h\in G$, and their products $\prod_{i\in I}R^{(i)}_g\times\prod_{i\in I}R^{(i)}_h\rightarrow\prod_{i\in I}R^{(i)}_{g+h}$ for $g,h\in G$ define the multiplication of $R$.

\smallskip

B) Let $R$ be a $G$-graded ring. The category $\grmod^G(R)$ has products, but $\bullet_{\ps}$ does not necessarily commute with them. The product $M=\prod_{i\in I}M^{(i)}$ of a family $(M^{(i)})_{i\in I}$ of $G$-graded $R$-modules  in $\grmod^G(R)$ is a $G$-graded $R$-module as follows. Its components are $\prod_{i\in I}M^{(i)}_g$ for $g\in G$, so that its underlying additive group is $\bigoplus_{g\in G}\prod_{i\in I}M^{(i)}_g$. For $i\in I$, the $R$-action of $M^{(i)}$ is given by maps $R_g\times M^{(i)}_h\rightarrow M^{(i)}_{g+h}$ for $g,h\in G$, and their products $R_g\times\prod_{i\in I}M^{(i)}_h\rightarrow\prod_{i\in I}M^{(i)}_{g+h}$ for $g,h\in G$ define the $R$-action of $M$.
\end{no}

\begin{thm}\label{1.70}
a) If $R$ is a $G$-graded ring, then $$\bullet_{\ps}\colon\grmod^G(R)\rightarrow\grmod^H(R_{\ps})$$ has a left adjoint if and only if\/ $\ke(\psi)$ is finite, and then $$\bigl(\grmod^H(R_{\ps})\xrightarrow{\bullet^{\ps}}\grmod^G(R),\grmod^G(R)\xrightarrow{\bullet_{\ps}}\grmod^H(R_{\ps})\bigr)$$ is an adjunction with unit $\gamma'_{\psi}$ and counit $\delta'_{\psi}$.

b) $\bullet_{\ps}\colon\grann^G\rightarrow\grann^H$ has a left adjoint if and only if\/ $\ke(\psi)$ is finite, and then $$\bigl(\grann^H\xrightarrow{\bullet^{\ps}}\grann^G,\grann^G\xrightarrow{\bullet_{\ps}}\grann^H\bigr)$$ is an adjunction with unit $\gamma_{\psi}$ and counit $\delta_{\psi}$.
\end{thm}

\begin{proof}
If $\ke(\psi)$ is finite then $\gamma'_{\psi}$ (for modules) and $\gamma_{\psi}$(for rings) are defined (\ref{1.36}~B), \ref{1.26}~B)). In both cases it is straightforward to check that $\bullet^{\ps}$ is left adjoint to $\bullet_{\ps}$.

We prove now the converse statement for modules, analogously to \cite[2.5.3]{no2}. Suppose $\bullet_{\ps}$ has a left adjoint and thus commutes with products (\cite[2.1.10]{ks}). We consider the family of $G$-graded $R$-modules $(M^{(g)})_{g\in\ke(\psi)}$ with $M^{(g)}=R(-g)$ for $g\in G$, so that $e_g=1_R\in M^{(g)}_g\setminus 0$ for $g\in G$. For $h\in\ke(\psi)$ we denote by $\pi^{(h)}\colon\prod_{g\in\ke(\psi)}M^{(g)}\rightarrow M^{(h)}$ and $\rho^{(h)}\colon\prod_{g\in\ke(\psi)}M^{(g)}_{\ps}\rightarrow M^{(h)}_{\ps}$ the canonical projections. There is a unique morphism $\xi$ in $\grmod^H(R_{\ps})$ such that for $h\in\ke(\psi)$ the diagram $$\xymatrix@R10pt{(\prod_{g\in\ke(\psi)}M^{(g)})_{\ps}\ar[rr]^{\xi}\ar[rd]_{(\pi^{(h)})_{\ps}}&&\prod_{g\in\ke(\psi)}M^{(g)}_{\ps}\ar[ld]^{\rho^{(h)}}\\&M^{(h)}_{\ps}&}$$ in $\grmod^H(R_{\ps})$ commutes. This $\xi$ is an isomorphism since $\bullet_{\ps}$ commutes with products. Taking components of degree $0$ we get a commutative diagram $$\xymatrix@C20pt@R10pt{\bigoplus\limits_{f\in\ke(\psi)}\prod\limits_{g\in\ke(\psi)}M^{(g)}_f\ar[rr]^{\xi_h}_{\cong}\ar[rd]&&\prod\limits_{g\in\ke(\psi)}\bigoplus\limits_{f\in\ke(\psi)}M^{(g)}_f\ar[ld]\\&\prod\limits_{f\in\ke(\psi)}\prod\limits_{g\in\ke(\psi)}M^{(g)}_f&}$$ in $\ab$, where the unmarked morphisms are the canonical injections (\ref{1.50}~B)). For $g\in\ke(\psi)$ we set $x^g_g=e_g\in M^{(g)}_g\setminus 0$ and $x^g_f=0\in M^{(g)}_f$ for $f\in\ke(\psi)\setminus\{g\}$. If $g\in\ke(\psi)$ then $\{f\in\ke(\psi)\mid x^g_f\neq 0\}$ has a single element, so that $((x^g_f)_{f\in\ke(\psi)})_{g\in\ke(\psi)}\in\prod_{g\in\ke(\psi)}\bigoplus_{f\in\ke(\psi)}M^{(g)}_f$. If $f\in\ke(\psi)$ then $x^f_f=e_f\neq 0$, hence $(x^g_f)_{g\in\ke(\psi)}\neq 0$, implying $$\{f\in\ke(\psi)\mid(x^g_f)_{g\in\ke(\psi)}\neq 0\}=\ke(\psi).$$ As $\xi_h$ is an isomorphism it follows $$((x^g_f)_{g\in\ke(\psi)})_{f\in\ke(\psi)}\in\bigoplus_{f\in\ke(\psi)}\prod_{g\in\ke(\psi)}M^{(g)}_f.$$ Thus, $\ke(\psi)=\{f\in\ke(\psi)\mid(x^g_f)_{g\in\ke(\psi)}\neq 0\}$ is finite.

Finally, the converse statement for rings is obtained analogously by considering the algebra $K[G]$ of $G$ over a field $K$, furnished with its canonical $G$-graduation, and the family $(R^{(g)})_{g\in\ke(\psi)}$ of $G$-graded rings with $R^{(g)}=K[G]$ for $g\in G$. Denoting by $\{e_g\mid g\in G\}$ the canonical basis of $K[G]$ and considering the elements $e_g\in R^{(g)}_g\setminus 0$ for $g\in G$ we can proceed as above.
\end{proof}


\section{Application to injective modules}\label{sec2}

\noindent\textit{We keep the hypothesis of Section \ref{sec1}. The symbols $\bullet_{\ps}$ and $\bullet^{\ps}$ refer always to coarsening functors for graded modules over appropriate graded rings.}\smallskip

In this section we apply the foregoing generalities to the question on how injective graded modules behave under coarsening functors. A lot of work on this question, but mainly in case $H=0$, was done by N\v{a}st\v{a}sescu et al. (e.g.~\cite{dnro}, \cite{nas}, \cite{nro}).

\begin{prop}\label{1.78}
Let $R$ be a $G$-graded ring and let $M$ be a $G$-graded $R$-module. If $M_{\ps}$ is injective then so is $M$.
\end{prop}

\begin{proof}
Analogously to \cite[A.I.2.1]{no1}.
\end{proof}

The converse of \ref{1.78} does not necessarily hold; see \cite[A.I.2.6.1]{no1} for a counterexample with $G=\Z$ and $H=0$. But in \cite[3.3]{nro} it was shown that the converse does hold if $G$ is finite and $H=0$. We generalize this to the case of arbitrary $G$ and $H$ such that $\ke(\psi)$ is finite, and we moreover show that this is the best we can get without imposing conditions on $R$ and $M$. Our proof is inspired by \cite[3.14]{dnro}. We first need some remarks on injectives and cogenerators, and a (probably folklore) variant of the graded Bass-Papp Theorem; we include a proof for lack of reference.

\begin{no}\label{1.77}
A) A functor between abelian categories that has an exact left adjoint preserves injective objects (\cite[3.2.7]{pop}).

\smallskip

B) In an abelian category $\C$, a monomorphism with injective source is a section, and a section with injective target has an injective source (\cite[8.4.4--5]{ks}). If $\C$ fulfils AB4$^*$ then an object $A$ is an injective cogenerator if and only if every object is the source of a morphism with target $A^L$ for some set $L$ (\cite[5.2.4]{ks}). This implies (\cite[3.2.6]{pop}) that if $A$ is an injective cogenerator and $L$ is a nonempty set then $A^L$ is an injective cogenerator.

\smallskip

C) If $R$ is a $G$-graded ring then $\grmod^G(R)$ is abelian, fulfils AB5, and has a generator. Hence, it has an injective cogenerator (\cite[9.6.3]{ks}).

\smallskip

D) Let $R$ be a $G$-graded ring and let $M$ be a $G$-graded $R$-module. Analogously to \cite[X.1.8 Proposition 12]{a} one sees that $M$ is a cogenerator if and only if every simple $G$-graded $R$-module is the source of a nonzero morphism with target $M$. As $M$ is simple if $M_{\ps}$ is so, it follows that if $M$ is a cogenerator then so is $M_{\ps}$.
\end{no}

\begin{prop}\label{1.79}
A $G$-graded ring $R$ is noetherian\footnote{as a $G$-graded ring, i.e., ascending sequences of graded ideals are stationary} if and only if $E^{\oplus\N}$ is injective for every injective cogenerator $E$ in $\grmod^G(R)$.
\end{prop}

\begin{proof}
Analogously to \cite[4.1]{chase} one shows that $R$ is noetherian if and only if countable sums of injective $G$-graded $R$-modules are injective. Thus, if $R$ is noetherian then $E^{\oplus\N}$ is injective for every injective cogenerator $E$. Conversely, suppose this condition to hold, let $(M_i)_{i\in\N}$ be a countable family of injective $G$-graded $R$-modules, and let $E$ be an injective cogenerator in $\grmod^G(R)$ (\ref{1.77}~C)). For $i\in\N$ there exist a nonempty set $L_i$ and a section $M_i\rightarrowtail E^{L_i}$ (\ref{1.77}~B)). Let $L=\prod_{i\in\N}L_i$. For $i\in\N$ the canonical projection $L\twoheadrightarrow L_i$ induces a monomorphism $E^{L_i}\rightarrowtail E^L$. Composition yields a section $M_i\rightarrowtail E^L$ (\ref{1.77}~B)). Taking the direct sum over $i\in\N$ we get a section $j\colon\bigoplus_{i\in\N}M_i\rightarrowtail(E^L)^{\oplus\N}$. Now, $E^L$ is an injective cogenerator (\ref{1.77}~B)), so $(E^L)^{\oplus\N}$ is injective by hypothesis, and as $j$ is a section thus so is $\bigoplus_{i\in\N}M_i$ (\ref{1.77}~B)). By the first sentence of the proof this yields the claim.
\end{proof}

\begin{thm}\label{1.80}
$\ke(\psi)$ is finite if and only if $$\bullet_{\ps}\colon\grmod^G(R)\rightarrow\grmod^H(R_{\ps})$$ preserves injectivity for every $G$-graded ring $R$.
\end{thm}

\begin{proof}
Finiteness of $\ke(\psi)$ implies that $\bullet_{\ps}$ preserves injectivity by \ref{1.25}~C), \ref{1.70}~a) and \ref{1.77}~A). For the converse we suppose that $\bullet_{\ps}$ preserves injectivity for every $G$-graded ring and assume that $\ke(\psi)$ is infinite. Let $A$ be a non-noetherian ring and let $R$ be the $G$-graded ring with $R_0=A$ and $R_g=0$ for $g\in G\setminus 0$. Then, $R_{\ps}$ is the $H$-graded ring with $(R_{\ps})_0=A$ and $(R_{\ps})_h)=0$ for $h\in H\setminus 0$, and in particular non-noetherian. Let $E$ be a injective cogenerator in $\grmod^H(R_{\ps})$ (\ref{1.77}~C)). It holds $(E^{\ps})_{\ps}=E^{\oplus\ke(\psi)}$ (\ref{1.20}~C)), and this $H$-graded $R_{\ps}$-module is injective by \ref{1.77}~A), \ref{1.37}~a), \ref{1.25}~A) and the hypothesis. Now, infinity of $\ke(\psi)$, \ref{1.77}~B) and \ref{1.79} yield the contradiction that $R_{\ps}$ is noetherian.
\end{proof}

\begin{no}
If $\ke(\psi)$ is infinite and torsionfree we can construct more interesting examples of $G$-graded rings $R$ such that $\bullet_{\ps}$ does not preserve injectivity than in the proof of \ref{1.80}. Indeed, let $A$ be the algebra of $\ke(\psi)$ over a field, furnished with its canonical $\ke(\psi)$-graduation. Let $R$ be the $G$-graded ring with $R_g=A_g$ for $g\in\ke(\psi)$ and $R_g=0$ for $g\in G\setminus\ke(\psi)$, so that $R_{\ps}$ is the $H$-graded ring with $(R_{\ps})_0=A_{[0]}$ and $(R_{\ps})_h=0$ for $h\in H\setminus 0$. The invertible elements of $R$ are precisely its homogeneous elements different from $0$ (\cite[11.1]{gilmer}), so that the $G$-graded $R$-module $R$ is injective. If $g\in\ke(\psi)\setminus 0$ then $x=1+e_g\in A$ (where $e_g$ denotes the canonical basis element of $A$ corresponding to $g$) is a nonhomogeneous non-zerodivisor of $A_{[0]}$ (\cite[8.1]{gilmer}), hence free and not invertible. So, there is a morphism of $A_{[0]}$-modules $\langle x\rangle_{A_{[0]}}\rightarrow A_{[0]}$ with $x\mapsto 1$ that cannot be extended to $A_{[0]}$, and thus the $H$-graded $R_{\ps}$-module $R_{\ps}$ is not injective.
\end{no}

The above result can be used to show that graded versions of covariant right derived cohomological functors commute with coarsenings with finite kernel (cf.~\cite{cglc}).


\section{Application to Hom functors}

\noindent\textit{We keep the hypotheses of Section \ref{sec2}. Let $R$ be a $G$-graded ring and let $M$ be a $G$-graded $R$-module. If no confusion can arise we write $\hm{}{\bullet}{\sq}$ instead of\/ $\hm{\grmod^G(R)}{\bullet}{\sq}$ for the Hom bifunctor with values in $\ab$.}\smallskip

As a second application of the generalities in Section 1 we investigate when coarsening functors commute with covariant graded Hom functors. For $H=0$ a complete answer was given by G\'omez Pardo, Militaru and N\v{a}st\v{a}sescu (\cite{gpmn}, see also \cite{gpn}). We generalize their result to arbitrary $H$, leading to the astonishing observation that if a covariant graded Hom functor commutes with some coarsening functor with infinite kernel then it commutes with every coarsening functor.

As in \cite{gpmn}, the notion of a small module turns out to be important. We start by recalling it and then prove a generalization of \cite[3.1]{gpmn} on coarsening of small modules and of steady rings.

\begin{no}
Let $I$ be a set, let $\NN=(N_i)_{i\in I}$ be a family of $G$-graded $R$-modules and let $\iota_j\colon N_j\rightarrowtail\bigoplus_{i\in I}N_i$ denote the canonical injection for $j\in I$. The monomorphisms $\hm{}{M}{\iota_j}$ in $\ab$ for $j\in I$ induce a morphism $\lambda_I^M(\NN)$ in $\ab$ such that the diagram $$\xymatrix@R15pt@C40pt{&\hm{}{M}{\bigoplus_{i\in I}N_i}\ar@{ (->}[r]&\hm{}{M}{\prod_{i\in I}N_i}\\\hm{}{M}{N_j}\ar@{ >->}[r]\ar@{ >->}[ru]^{\hm{}{M}{\iota_j}}&\bigoplus_{i\in I}\hm{}{M}{N_i}\ar@{ (->}[r]\ar[u]_{\lambda_I^M(\NN)}&\prod_{i\in I}\hm{}{M}{N_i},\ar[u]^{\cong}}$$ where the unmarked monomorphisms are the canonical injections and the unmarked isomorphism is the canonical one, commutes for $j\in I$. It follows that $\lambda_I^M(\NN)$ is a monomorphism. If $\NN$ is constant with value $N$ then we write $\lambda_I^M(N)$ instead of $\lambda_I^M(\NN)$. Varying $\NN$ we get a monomorphism $$\lambda_I^M\colon\bigoplus_{i\in I}\hm{}{M}{\bullet_i}\rightarrowtail\hm{}{M}{\bigoplus_{i\in I}\bullet_i}$$ of covariant functors from $\grmod^G(R)^I$ to $\ab$. If $I$ is finite then $\lambda_I^M$ is an isomorphism.
\end{no}

\begin{no}
A) If $N$ is a $G$-graded $R$-module then $M$ is called \textit{$N$-small} if $\lambda_I^M(N)$ is an isomorphism for every set $I$. Furthermore, $M$ is called \textit{small} if $\lambda_I^M$ is an isomorphism for every set $I$, and this holds if and only if $M$ is $N$-small for every $G$-graded $R$-module $N$ (\cite[1.1 i)]{gpmn}).

\smallskip

B) If $M$ is of finite type then it is small. The $G$-graded ring $R$ is called \textit{steady} if every small $G$-graded $R$-module is of finite type. Noetherian $G$-graded rings are steady, but the converse does not necessarily hold (\cite[3.5]{gpmn}, \cite[7$^{\circ}$; 10$^{\circ}$]{rentschler}). Furthermore, for every group $G$ there exists a $G$-graded ring that is not steady (\cite[p.~3178]{gpmn}).
\end{no}

\begin{prop}\label{2.30}
a) $M$ is small if and only if $M_{\ps}$ is small.

b) If $N$ is a $G$-graded $R$-module, then $M$ is $\bigoplus_{g\in G}N(g)$-small if and only if $M_{\ps}$ is $N_{\ps}$-small.
\end{prop}

\begin{proof}
Immediately from \ref{1.37}~a), \cite[1.3 i)--ii)]{gpmn}, and the facts that $\bullet_{\ps}$ and $\bullet^{\ps}$ commute with direct sums and that $\bigoplus_{g\in G}N(g)=(N_{\ps})^{\ps}$.
\end{proof}

\begin{prop}
If $R_{\ps}$ is steady then so is $R$; the converse holds if\/ $\ke(\psi)$ is finite.
\end{prop}

\begin{proof}
If $R_{\ps}$ is steady and $N$ is a small $G$-graded $R$-module then $N_{\ps}$ is small (\ref{2.30}~a)), hence of finite type, and thus $N$ is of finite type, too. Conversely, suppose $\ke(\psi)$ is finite and $R$ is steady, and let $N$ be a small $H$-graded $R_{\ps}$-module. Since $\bullet_{\ps}$ commutes with direct sums it follows that $N^{\ps}$ is small (\ref{1.70}~a), \cite[1.3 ii)]{gpmn}), hence of finite type, and thus $(N^{\ps})_{\ps}$ is of finite type, too. The canonical epimorphism $\beta'_{\psi}(N)\colon(N^{\ps})_{\ps}\twoheadrightarrow N$ (\ref{1.36}~B)) shows now that $N$ is of finite type.
\end{proof}

Next, we look at graded covariant Hom functors and characterize when they commute with coarsening functors, thus generalizing \cite[3.4]{gpmn}.

\begin{no}
The $G$-graded covariant Hom functor $\grhm{G}{R}{M}{\bullet}$ maps a $G$-graded $R$-module $N$ onto the $G$-graded $R$-module $$\grhm{G}{R}{M}{N}=\bigoplus_{g\in G}\hm{\grmod^G(R)}{M}{N(g)}.$$ For a $G$-graded $R$-module $N$ and $g\in G$ we have a monomorphism $$\hm{\grmod^G(R)}{M}{N(g)}\rightarrowtail\hm{\grmod^H(R_{\ps})}{M_{\ps}}{N_{\ps}(\psi(g))},\;u\mapsto u_{\ps}$$ in $\ab$, inducing a monomorphism $$h_{\psi}(M,N)\colon\grhm{G}{R}{M}{N}_{\ps}\rightarrowtail\grhm{H}{R_{\ps}}{M_{\ps}}{N_{\ps}}$$ in $\grmod^H(R{\ps})$. Varying $N$ we get a monomorphism $$h_{\psi}^M\colon\grhm{G}{R}{M}{\bullet}_{\ps}\rightarrowtail\grhm{H}{R_{\ps}}{M_{\ps}}{\bullet_{\ps}}.$$
\end{no}

\begin{lemma}\label{2.60}
If $N$ is a $G$-graded $R$-module such that $h_{\psi}^M(N)$ is an isomorphism then $\lambda_{\ke(\psi)}^M((N(g))_{g\in\ke(\psi)})$ is an isomorphism.
\end{lemma}

\begin{proof}
Let $u\colon M\rightarrow\bigoplus_{g\in G}N(g)$ in $\grmod^G(R)$. As $$(\bigoplus_{g\in\ke(\psi)}N(g))_{\ps}=\bigoplus_{g\in\ke(\psi)}N_{\ps}$$ we can consider the codiagonal $\bigoplus_{g\in\ke(\psi)}N_{\ps}\rightarrow N_{\ps}$ in $\grmod^H(R_{\ps})$ as a morphism $\nabla\colon(\bigoplus_{g\in\ke(\psi)}N(g))_{\ps}\rightarrow N_{\ps}$. Composition with $u_{\ps}$ yields $\nabla\circ u_{\ps}\in\grhm{H}{R_{\ps}}{M_{\ps}}{N_{\ps}}_0$. By our hypothesis there exist a finite subset $E\subseteq\ke(\psi)$ and for $e\in E$ a $v_e\colon M\rightarrow N(e)$ in $\grmod^G(R)$ such that for $x\in M$ it holds $\nabla(u(x))=\sum_{e\in E}^rv_e(x)$. For $g\in G$ this implies $$\nabla(u(M_g))\subseteq\sum_{e\in E}N_{g+e}=\sum_{e\in E}N(e)_g.$$ Let $x=(x_g)_{g\in G}\in M$ with $x_g\in M_g$ for $g\in G$. For $g\in G$ there exists $(n^{(g)}_h)_{h\in\ke(\psi)}\in(\bigoplus_{h\in\ke(\psi)}N(h))_g=\bigoplus_{h\in\ke(\psi)}N(h)_g$ with $u(x_g)=(n^{(g)}_h)_{h\in\ke(\psi)}$, but it holds $\nabla(u(x_g))\in\sum_{e\in E}N(e)_g$ and therefore $n^{(g)}_h=0$ for $h\in\ke(\psi)\setminus E$. This implies $\nabla(u(x))\in\bigoplus_{e\in E}N(e)$, thus $u(M)\subseteq\bigoplus_{e\in E}N(e)$, and hence the claim.
\end{proof}

\begin{thm}\label{2.70}
$h_{\psi}^M$ is an isomorphism if and only if $M$ is small or $\ke(\psi)$ is finite.
\end{thm}

\begin{proof}
If $\ke(\psi)$ is finite then this is readily seen to hold. We suppose $\ke(\psi)$ is infinite. If $M$ is small then $h_{0}^M$ is an isomorphism (\cite[3.4]{gpmn}), thus $h_{\ps}^M$ is an isomorphism, too. Conversely, we suppose $h_{\ps}^M$ is an isomorphism and prove that $M$ is small. Let $(L_i)_{i\in\N}$ be a family of $G$-graded $R$-modules, let $L=\bigoplus_{i\in\N}L_i$, let $l_i\colon L_i\rightarrowtail L$ denote the canonical injection for $i\in\N$, and let $f\colon M\rightarrow L$ in $\grmod^G(R)$. Let $N=\bigoplus_{g\in\ke(\psi)}L(g)$, and let $n_g\colon L(g)\rightarrowtail N$ denote the canonical injection for $g\in\ke(\psi)$. It is readily checked that $N(g)\cong N$ for $g\in\ke(\psi)$. As $\ke(\psi)$ is infinite we can without loss of generally suppose $\N\subseteq\ke(\psi)$. Choosing for $g\in\N$ an isomorphism $N\rightarrow N(g)$ we get a monomorphism $v\colon N^{\oplus\N}\rightarrowtail\bigoplus_{g\in\ke(\psi)}N(g)$. Furthermore, we get a monomorphism $u=\bigoplus_{i\in\N}n_0\circ l_i\colon\bigoplus_{i\in\N}L_i\rightarrowtail N^{\oplus\N}$, hence a morphism $v\circ u\circ f\colon M\rightarrow\bigoplus_{g\in\ke(\psi)}N(g)$. By \ref{2.60} there exists a finite subset $E\subseteq\ke(\psi)$ with $v(u(f(M)))\subseteq\bigoplus_{g\in E}N(g)\subseteq\bigoplus_{g\in\ke(\psi)}N(g)$. By construction of $v$ there exists a finite subset $E'\subseteq\N$ with $u(f(M))\subseteq N^{\oplus E'}\subseteq N^{\oplus\N}$. Thus, by construction of $u$, we have $f(M)\subseteq\bigoplus_{i\in E'}L_i\subseteq\bigoplus_{i\in\N}L_i=L$. Therefore, $M$ is small.
\end{proof}

At the end we get the surprising corollary mentioned before.

\begin{cor}\label{2.80}
If there exists an infinite subgroup $F\subseteq G$ such that, denoting by $\pi\colon G\twoheadrightarrow G/F$ the canonical projection, $h_{\pi}^M$ is an isomorphism, then $h_{\psi}^M$ is an isomorphism for every epimorphism $\psi\colon G\twoheadrightarrow H$ in $\ab$, and in particular $h_{0}^M$ is an isomorphism
\end{cor}

\begin{proof}
Immediately from \ref{2.70}.
\end{proof}


\end{document}